\documentclass[a4paper]{amsart}

\usepackage{amsmath, amssymb, amsthm,a4wide} 

\theoremstyle{plain}                       
\newtheorem{theorem}{Theorem}

\newtheorem{corollary}[theorem]{Corollary} 

\theoremstyle{remark}                      
\newtheorem*{acknow}{Acknowledgement}          

\newcommand\pair[1]{\langle\, #1\,\rangle}
\newcommand\ppair[1]{(\,#1\,)}

\DeclareMathOperator\linsp{span}

\newcommand{\lt}{L}                        
\newcommand{\rt}{R}

\newcommand\col{\colon}
\newcommand\sub{\subseteq}

\newcommand\cstar{\mathrm{C}^*}

\newcommand{\A}{\mathrm{A}}              
\newcommand{\B}{\mathrm{B}}              

\newcommand{\set}[2]{\{\,{\textstyle#1};\,{\textstyle #2}\,\}}

\newcommand{\inv}{^{-1}}                   
\newcommand\conj{\overline}                

\newcommand\lone{\mathrm{L}^1}
\newcommand\ltwo{\mathrm{L}^2}
\newcommand\linfty{\mathrm{L}^\infty}
\newcommand{\vn}{\mathrm{VN}}
\newcommand{\C}{\mathrm{C}}
\newcommand{\G}{\mathbb{G}}

\newcommand{\ot}{\otimes}
\newcommand{\cop}{\Delta}
\newcommand{\id}{\mathrm{id}}
\newcommand{\M}{\mathrm{M}}

\newcommand\dual[1]{\widehat{#1}} 

\newcommand\conv{\star}

\newcommand\T{\mathbb{T}}
\renewcommand\H{\mathbb{H}}

\newcommand\Z{\mathbb{Z}}
\newcommand\complex{\mathbb{C}}

\newcommand\SU{\mathrm{SU}}
\newcommand\U{\mathrm{U}}
\newcommand\SO{\mathrm{SO}}

\begin{document}
\title[Idempotent states]%
      {Idempotent states on locally compact groups and quantum groups}

\author{Pekka Salmi}

\dedicatory{Dedicated to Professor Victor Shulman on the occasion of
  his 65th birthday}

\address{Department of Mathematical Sciences, University of Oulu,
         PL~3000, FI-90014 Oulun yliopisto, Finland}

\email{pekka.salmi@iki.fi}


\begin{abstract}
This is a short survey on idempotent states on locally compact groups 
and locally compact quantum groups.
The central topic is the relationship between 
idempotent states, subgroups and invariant C*-subalgebras. 
We concentrate on recent results on locally compact quantum groups,
but begin with the classical notion of idempotent probability measure. 
We also consider the `intermediate' case of idempotent 
states in the Fourier--Stieltjes algebra: this is the dual case 
of idempotent probability measures and so an instance of 
idempotent states on a locally compact quantum group.
\end{abstract}

\maketitle
This is a short survey on idempotent states on locally compact groups 
and, more generally, on locally compact quantum groups. 
Idempotent states arise for example as 
limits of random walks (as we shall see in section~\ref{sec:rw})
and as limits in ergodic theorems for random walks
\cite{franz-skalski:ergodic}.  
Idempotent states are also connected to the construction 
of the Haar measure of a compact group:
taking the Ces\`aro limit 
of convolution powers of a probability measure  
gives an idempotent probability measure, 
which is the Haar measure if the original measure is suitably 
chosen. The same process works in the case of compact quantum groups:
this is the construction of Haar state due to Woronowicz 
\cite{woronowicz:compact-matrix-pseudogroups,woronowicz:compact-quantum}. 
Very recently, idempotent states on locally compact quantum 
groups have also turned up in connection with Hopf images 
\cite{banica-franz-skalski:id-states-hopf-images} 
and Poisson boundaries \cite{kalantar-neufang-ruan:poisson}.

As we shall see, idempotent states are inherently related to subgroups. 
However, there is some evidence \emph{against} the preceding claim,
such as Pal's example of an idempotent state on the 
Kac--Paljutkin quantum group that does not arise from the Haar state 
of a subgroup \cite{pal:idempotent-state}. In this paper we shall see
that in fact also Pal's example is associated with a subgroup (in 
a different way), and so perhaps there is still hope to 
connect all idempotent states to subgroups, quotient groups or
combinations of these. 
There are not many new things in this survey article: 
only a cute new proof to a known result 
and the already-mentioned insight to the example of Pal.   
Many of the results, and much more, can be found in the 
recent papers due to (combinations of) 
Franz, Skalski, Tomatsu and the author 
\cite{franz-skalski:idempotent-states,
  franz-skalski:algebraic-idem,franz-skalski-tomatsu:idem,
  salmi:compact-subgroups,
  salmi-skalski:idempotent-states}.

\section{Random walks and idempotent probability measures}
\label{sec:rw}

Every probability measure $\mu$ on a discrete group $G$ 
determines a random walk:
if we start from point $s\in G$, then the probability for 
taking a step to $t\in G$ is 
\[
P(s\mapsto t) = \mu_{ts\inv},
\]
where $\mu_{ts\inv} := \mu(\{ts\inv\})$.
Suppose that we start from the identity $e$ and $X_k$ is 
the random variable denoting the position after $k$ steps.
We can use convolution to describe the random walk:
\begin{align*}
P(X_1 = t) &= \mu_t\\
P(X_2 = t) &= \sum_{s\in G}P(X_1 = s)P(s\mapsto t)  
               = \sum_{s\in G} \mu_s \mu_{ts\inv} = (\mu\conv\mu)_t \\
&\;\;\vdots\\
P(X_k = t) &= (\mu^{\conv k})_t.
\end{align*}
In general the convolution of measures $\mu$ and $\nu$ 
on a locally compact group $G$ is defined by 
\[
\pair{\mu\conv \nu, f} = \iint f(st)\,d\mu(s)\,d\nu(t)
\qquad (f\in \C_0(G));
\]
here, and throughout the paper, we consider measures on $G$ as
functionals on the C*-algebra $\C_0(G)$  
of continuous functions on $G$ vanishing at infinity.
In the discrete case the convolution boils down to 
\[
\mu\conv\nu = \sum_{s,t\in G}  \mu_s\nu_t\delta_{st}
        = \sum_{s\in G}\biggl(\sum_{t\in G} \mu_{st\inv}\nu_{t}\biggr)\delta_s
\]
where $\delta_s$ denotes the Dirac measure at $s$. 

More generally, a probability measure on a locally compact group
determines a random walk on that group. 
From this point of view, we shall see that random walks give rise to 
idempotent probability measures. An \emph{idempotent probabity measure} on
a locally compact group $G$  is a probability measure $\mu$ on $G$ that
is an idempotent under the convolution product:
\[
\mu\conv \mu = \mu.
\]
Now suppose that $G$ is a compact group and 
$\nu$ is a probability measure on $G$. 
If the sequence of convolution powers of $\nu$ converges 
in the weak* topology, then the limit is an idempotent probability
measure, which embodies the limit of the random walk.
The convergence of such a sequence of convolution powers 
is widely studied in probability theory (see for example 
\cite{grenander:probabilities-alg,
hognas-mukherjea:probability-measures-on-semigroups}).

Taking a slightly different approach, consider the Ces\`aro averages
\[
\frac1n \sum_{k=1}^n \nu^{\conv k}.
\]
The sequence of these averages always converges in the weak* topology. 
Morevover, the limit $\mu$ is an idempotent probability measure. 
This gives a useful way to generate idempotent probability measures:
for example, the Haar measure of a compact group may be constructed 
this way. 
(In the case of non-compact groups the situation is more complicated 
and to make sure that the limit is non-zero, some form of 
tightness needs to be assumed for the sequence of convolution powers
of $\nu$.)

\section{Kawada--It\^o theorem} \label{sec:kawada-ito}

Now that we have seen how idempotent probability measures may arise 
in practice, a natural question is how to characterise these measures. 
In the case of locally compact abelian groups,
we may use the Fourier--Stieltjes transform to 
convert an idempotent measure to a characteristic function 
on the dual group. This trivialises the algebraic side of things, but
now the positivity and  the normalisation condition become non-trivial. 
Still, this is a useful approach and, as we shall see in the next
section, leads to a simple characterisation. 
But let us first review the history of the general problem. 

Already in 1940 Kawada and It\^o \cite{kawada-ito:idem}
characterised idempotent probability measures 
on compact groups as the normalised Haar measures of
compact subgroups. It seems that harmonic analysts were 
unaware of this paper, and Wendel \cite{wendel:idem} rediscovered the
result in 1954 (truth be told, Wendel's main result is an interesting new
proof for the existence of Haar measure on a compact group $G$,
using idempotents in the compact semigroup formed by the probability 
measures on $G$). Trying to characterise all idempotent measures on a 
locally compact abelian group, Rudin~\cite{rudin:idem, rudin:measure-algebras}
showed that any idempotent measure is concentrated on
a compact subgroup, thereby extending Wendel's result 
-- or that of Kawada--It\^o -- to locally compact abelian
groups. (The full description  of idempotent measures on locally
compact abelian groups is due to Cohen~\cite{cohen:idempotents}; the
non-abelian case is open.) 
Independently, Pym~\cite{pym:idem} and Loynes~\cite{loynes:idem}
characterised idempotent probability measures on 
locally compact groups (not necessarily abelian)
as the normalised Haar measures of compact subgroups. 
However, since the problem took two separate paths, 
one starting from Kawada--It\^o and another one from Wendel,
it is perhaps not that surprising that the problem was solved 
already in 1954, by Kelley \cite{kelley:averaging}. 
The three approaches, due to Kelley, Pym and Loynes, are all quite
different: Kelley studies operators on $\C_0(G)$, 
Pym idempotent measures on semigroups and Loynes operator-valued
Fourier transform. There are also other generalisations, for example
one due to Parthasarathy \cite{parthasarathy:idem} 
to complete separable metric groups, which need not be locally compact.

\begin{theorem}[Kawada--It\^o]
Let $\mu$ be an idempotent probability measure 
on  a locally compact group $G$.  
Then there is a compact subgroup $H$ of $G$ 
such that $\mu$ is the normalised Haar measure of $H$
\textup{(}considered as a measure on $G$\textup{)}.
\end{theorem}

The Kawada--It\^o theorem gives a procedure to 
construct the Haar measure of a compact group. 
Start with a probability measure $\nu$ whose support 
generates the compact group $G$. As mentioned in the preceding 
section, the Ces\`aro averages of convolution powers of $\nu$ 
converge in the weak* topology to an idempotent probability measure. 
It follows from the Kawada--It\^o theorem, 
that the limit is the normalised Haar measure 
of $G$ due to the choice of $\nu$. 

As another example, 
consider the circle group $\T$ and the Dirac measure $\delta_z$ at 
some $z\in\T$. Then there is an idempotent probability measure $\mu$ 
such that $\frac1n \sum_{k=1}^n \delta_z^{*k}\to \mu$ weak*. 
If $z$ is a rational multiple of $\pi$, then $\mu$ 
is the counting measure of the finite subgroup of $\T$
generated by $z$.
On the other hand, if $z$ is an irrational multiple of $\pi$,
then $\mu$ is the normalised Lebesgue measure on $\T$.
The latter statement amounts to the Weyl equidistribution theorem.

\section{Idempotent states in the Fourier--Stieltjes algebra}

Let $G$ be a locally compact group. 
The \emph{Fourier--Stieltjes algebra} $\B(G)$  is the collection of
all coefficient functions $\ppair{\pi(\cdot)\xi\mid \zeta}$ 
of strongly continuous unitary representations $\pi$ of $G$ 
(here $\xi$ and $\zeta$ are elements of the representation space
$\mathcal{H}_\pi$).  
The Fourier--Stieltjes algebra is the dual space of 
the universal group C*-algebra $\cstar(G)$ (this determines the 
norm of $\B(G)$), and $\B(G)$ is a Banach algebra under 
the pointwise multiplication of functions.
If $G$ is abelian, $\B(G)$ is isomorphic, via Fourier--Stieltjes 
transform, to the measure algebra $\M(\dual G)$ 
of Radon measures on the dual group $\dual G$. 

An \emph{idempotent state} in $\B(G)$ is 
a state on $\cstar(G)$ that is an idempotent:
\[
u^2 = u.
\]
That $u$ is a state means that $u$ is a positive definite function
with $u(e) = 1$, where $e$ denotes the identity of $G$.
In the case of abelian $G$, the Fourier--Stieltjes transform 
takes an idempotent probability measure on $\dual G$ 
to an idempotent state in $\B(G)$. 
This explains why the characterisation of idempotent
states in $\B(G)$ may be viewed as the dual version 
of the Kawada--It\^o theorem. In fact, many early 
studies on idempotents in $\M(G)$, with $G$ abelian, 
used the Fourier--Stieltjes transform to translate the 
problem to the dual setting.  

Continuing from the work of Cohen \cite{cohen:idempotents},
mentioned in the previous section, Host~\cite{host:idempotents}
characterised all idempotents in $\B(G)$ as characteristic functions
of sets in the open coset ring of $G$. 
However, his characterisation does not immediately lead 
to the following characterisation of idempotent states,
which is due to Ilie and Spronk~\cite{ilie-spronk:cb-homos}.
The short proof presented here is new (Ilie and Spronk obtained
their result as a corollary of a more general characterisation 
of contractive idempotents).

\begin{theorem}
Every idempotent state in $\B(G)$ is 
the characteristic function of an open subgroup of $G$.
\end{theorem}

\begin{proof}
First of all, every idempotent in $\B(G)$ is a characteristic function 
of some open (and closed) set $H$, because $\B(G)$ consists of 
continuous functions. 
Denote the universal representation of $G$ by $\varpi$. 
If $s\in H$, then 
\[
u(s\inv) = \pair{u,\varpi(s\inv)} 
      = \pair{u,\varpi(s)^*} = \conj{\pair{u,\varpi(s)}} = 1,
\]
so $s\inv$ is also in $H$. Moreover,
\[
\pair{u,\varpi(s)^*\varpi(s)} = u(e) = 1,
\]
because $u$ is a state, and
\[
\pair{u,\varpi(s)^*}\pair{u,\varpi(s)} = 1\cdot 1 = 1. 
\]
It follows from Choi's theorem on multiplicative domains
(see \cite[Theorem 3.19]{paulsen:cb-maps}) 
that $u$ is multiplicative at $\varpi(s)$. So for every $t\in G$,
\[
u(ts) = u(t)u(s)
\]
which implies that $H$ is closed under multiplication.
Hence $H$ is an open subgroup and $u = 1_H$.
\end{proof}

The fact that compact subgroups in the Kawada--It\^o theorem 
have changed to open ones in the preceding result
reflects subgroup duality.
Suppose that $G$ is abelian and $H$ is a closed subgroup 
of $G$. Then the continuous characters on $G$ that are constant $1$ on
$H$ form a closed subgroup $H^\perp$ of $\dual G$.
Now $H$ is compact if and only if $H^\perp$ is open, and vice versa. 
Moreover, the Fourier--Stieltjes transform maps a measure on $G$
supported by $H$ to a function on $\dual G$ 
that is constant on the cosets of $H^\perp$.

\section{Locally compact quantum groups}

Locally compact quantum groups provide a 
natural context to discuss the results in the previous 
sections in a unified manner. We shall walk through the definition
due to Kustermans and Vaes \cite{kustermans-vaes:lcqg}.
Let $\G$ denote a locally compact quantum group.
This means that we have a C*-algebra $\C_0(\G)$,
a non-degenerate $*$-homomorphism 
$\cop\col\C_0(\G)\to \M(\C_0(\G)\ot \C_0(\G))$
(where the tensor product is the minimal C*-algebraic tensor product 
and $\M({}\cdot{})$ denotes the multiplier algebra) such that 
\[
(\id\otimes \cop)\cop = (\cop\otimes \id)\cop 
   \qquad\text{(coassociativity)};
\]
and     
\[
\overline{\linsp}\,\cop(\C_0(\G))(\C_0(\G)\otimes 1) = 
     \overline{\linsp}\,\cop(\C_0(\G))(1\otimes \C_0(\G)) =
     \C_0(\G)\otimes \C_0(\G).
\]
The map $\cop$ is called the \emph{comultiplication} of $\G$.
We also need to assume that there exist
left and right Haar weights on $\C_0(\G)$, denoted by $\phi$ and 
$\psi$, respectively. These are so-called KMS-weights, which are
densely defined,  faithful and lower semicontinuous. The important
invariance  properties, as with Haar measures, are that 
\[
\phi\bigl((\omega\otimes \id)\cop(a)\bigr) = \omega(1)\phi(a)
\]
and
\[
\psi\bigl((\id\otimes \omega)\cop(b)\bigr) = \omega(1)\psi(b)
\]
whenever $\omega\in \C_0(\G)^*_+$ 
and $a,b\in \C_0(\G)_+$ are such that $\phi(a)<\infty$
and $\psi(b)<\infty$.
So a locally compact quantum group is given by a C*-algebra
that has a suitable comultiplication and left and right 
Haar weights. It is convenient to use the suggestive notation 
$\C_0(\G)$ for the C*-algebra because in the commutative case 
the C*-algebra \emph{is} $\C_0(G)$ for some locally compact group $G$.
In this case the comultiplication is given by 
dualised group multiplication:
\[
\cop(f)(s,t) = f(st)\qquad (f\in \C_0(G), s,t\in G). 
\]
Note that $\cop(f)\in \C_b(G\times G) = \M(\C_0(G)\otimes \C_0(G))$
but $\cop(f)\notin \C_0(G)\otimes \C_0(G)$ unless $G$ is 
compact or $f=0$.
Of course the left and right Haar weights are given
by integration against the left and right Haar measures, respectively.
Whenever $\G$ is a locally compact quantum group
such that $\C_0(\G)$ is commutative, it is of this form.

Next we consider the dual of the commutative case, which 
is known as the co-commutative case. 
Let $\lambda$ be the left regular representation of $G$.
Then the reduced group C*-algebra $\cstar_r(G)$ is generated by 
$\lambda(\lone(G))$ in $\B(\ltwo(G))$.
(Hopefully the reader is not confused by the two uses of `$\B$'
as both the Fourier--Stieltjes algebra and 
the algebra of bounded operators; 
the distinction should be clear from the context.)
We define a comultiplication on $\cstar_r(G)$ by putting
\[
\cop(\lambda(s)) = \lambda(s)\otimes \lambda(s)\qquad(s\in G).
\]
Note that actually $\lambda(s)$ is in $\M(\C_r^*(G))$ 
but the above does define a unique comultiplication 
on $\cstar_r(G)$, because the linear span of $\lambda(G)$ is 
strictly dense in $\M(\C_r^*(G))$. 
In this case the left and right Haar weights coincide 
and are the so-called Plancherel weight. 
The construction of this weight uses Tomita--Takesaki theory
\cite[section~VII.3]{takesaki:vol2} 
(for discrete $G$, $\phi(a) = \ppair{a\delta_e\mid\delta_e}$
is the usual tracial state).

Both these examples may be considered as Kac algebras 
\cite{enock-schwartz:kac}. Every Kac algebra 
determines a locally compact quantum group, so the 
latter notion is more general. For an example of a locally compact 
quantum group that is not a Kac algebra, see section~\ref{sec:SU},
which includes the description of the quantum deformation of $\SU(2)$
defined by Woronowicz. For a more thorough introduction
to locally compact quantum groups, see for example the 
survey by Kustermans and Tuset \cite{kus-tuset:survey1,kus-tuset:survey2} 
or the book by Timmermann \cite{timmermann:quantum-groups}.

From now on we shall concentrate on locally compact quantum groups $\G$ 
that are \emph{coamenable}. That means
that there is a state $\epsilon$ on $\C_0(\G)$, called the
\emph{counit} of $\G$, such that 
\[
(\id\ot\epsilon)\cop(a) = (\epsilon\ot\id)\cop(a) = a 
\]
for every $a\in\C_0(\G)$. In the commutative case coamenability 
is a vacuous condition (every commutative quantum group is coamenable),
but a co-commutative quantum group $\G = \dual G$ is coamenable 
if and only if the locally compact group $G$ is amenable.

\section{Classical cases as instances of idempotent states 
         on locally compact quantum groups} 

The notion of idempotent state from the two classical cases 
-- idempotent probability measures on groups 
and idempotent states in the Fourier--Stieltjes algebra -- 
is easily generalised to the language of locally compact quantum
groups. The dual space of the C*-algebra $\C_0(\G)$ carries 
a natural Banach algebra structure: the multiplication is defined 
by 
\[
\omega\conv \sigma(a) = (\omega\ot\sigma)\cop(a)\qquad
           (\omega,\sigma\in \C_0(\G)^*, a\in \C_0(\G)).
\]
An \emph{idempotent state} on a locally compact quantum group $\G$ 
is a state $\omega$ on the C*-algebra $\C_0(\G)$ that is an 
idempotent under the product defined above: 
$\omega\conv\omega = \omega$. 

A much more difficult task than the definition is to unify 
the results from the classical cases to general results on 
locally compact quantum groups. Indeed, it is perhaps 
not even possible to do so. To even bring forth this discussion we
need some  further terminology. 

A locally compact quantum group $\H$ is \emph{compact}
if $\C_0(\H)$ is unital, in which case we write $\C(\H)$ for 
$\C_0(\H)$. A \emph{compact quantum subgroup} of a coamenable
locally compact quantum group $\G$ is a compact quantum group $\H$ 
such that there exists a surjective $*$-homomorphism 
\[
\pi\col \C_0(\G)\to \C(\H)\qquad (\pi\ot \pi)\cop_\G = \cop_\H\pi.
\]
(The reader is warned that there are other definitions of closed
quantum subgroup
\cite{vaes:imprimitivity,vaes-vainerman:low-dimensional}, and it is
not clear whether they are all equivalent. 
In our situation, all the definitions coincide as they do in 
many other cases; see \cite{dkss:subgroups}.)
A compact quantum subgroup of $\G$ always gives rise to an 
idempotent state on $\G$. Indeed, when a locally compact quantum group $\H$
is compact, the left and right Haar weights are actually bounded
functionals and coincide. By normalisation, there exists 
a unique state -- the \emph{Haar state} -- $\phi_\H$ on $\C(\H)$ 
that is both left and right invariant.
Using the subgroup morphism $\pi$, we may pull back $\phi_\H$ 
to obtain an idempotent state $\omega = \phi_\H\pi$ on $\G$. 
Obviously $\omega$ is a state and it is an idempotent 
due to invariance of $\phi_\H$:
\begin{align*}
\omega\conv\omega &= \bigl((\phi_\H\pi)\otimes(\phi_\H\pi)\bigr)
                                                           \cop_\G
              =\bigl(\phi_\H\otimes\phi_\H)(\pi\otimes\pi)\cop_\G\\
              &= (\phi_\H\otimes\phi_\H)\cop_\H\pi
               = \phi_\H(1)\phi_\H\pi = \omega.
\end{align*}
It should be noted that in the case of compact quantum groups,
the existence of Haar state follows from the 
other axioms as shown by 
Woronowicz~\cite{woronowicz:compact-matrix-pseudogroups,
woronowicz:compact-quantum}. 
Indeed, this may be done with a similar process of
using Ces\`aro averages as mentioned after the Kawada--It\^o theorem:
the Ces\`aro averages of convolution powers of a faithful state
converge to the Haar state (however, the resulting Haar state is not
necessarily faithful). 
The assumption that there is a faithful state on the C*-algebra,
which is true in the separable case,
may be dropped, as shown by Van Daele \cite{vandaele:haar-compact}.

Now it is possible to at least formulate the statement of 
the Kawada--It\^o theorem: every idempotent 
state on a locally compact quantum group $\G$ is 
a \emph{Haar idempotent}, that is,
of the form $\omega = \phi_\H\pi$  where $\phi_\H$ is 
the Haar state of a compact quantum subgroup $\H$ of $\G$
and $\pi\col  \C_0(\G)\to \C(\H)$ is the associated morphism.
The problem is that this statement is false. It is, moreover, 
easily seen to be false. Let $G$ be an amenable locally compact group
with a non-normal open subgroup $H$. Then $1_H$ is an 
idempotent state on $\C_r^*(G)$. If $1_H$ were a pullback of 
the Haar state of a compact quantum subgroup of $\C_r^*(G)$,
then the compact quantum subgroup would necessarily be 
of the form $\C_r^*(G/H)$. But $H$ not being normal, this is not
possible (see \cite[Theorem~6.2]{franz-skalski:algebraic-idem} for the 
finite case and \cite[section~7]{salmi:compact-subgroups} 
for a related discussion on left invariant C*-subalgebras 
in $\C_r^*(G)$). 

Although this example, obtainable with finite groups, certainly 
seems to be the most straightforward counterexample of 
the Kawada--It\^ o theorem for quantum groups, it was not the 
first one. The first counterexample is due to Pal 
\cite{pal:idempotent-state} and it comes from a genuine quantum group:
the Kac--Paljutkin quantum group. We shall next describe this 
quantum group  and Pal's example as well as provide a new insight 
to this example.

\section{Pal's counterexample} 

The underlying C*-algebra of the Kac--Paljutkin quantum group $\G$ is 
\[
\complex\oplus\complex\oplus\complex\oplus\complex\oplus \M_2(\complex),
\] 
the basis of which is given by the vectors 
\[
e_k = \delta_{1,k}\oplus\delta_{2,k}\oplus\delta_{3,k}\oplus\delta_{4,k}\oplus
\begin{pmatrix}
\delta_{5,k} & \delta_{8,k}\\
\delta_{7,k} & \delta_{6,k} 
\end{pmatrix}
\]
$k = 1, 2, \ldots, 8$. The comultiplication of $\G$ is defined by
\begin{align*}
\cop(e_1)&= e_1\ot e_1 + e_2\ot e_2 + e_3\ot e_3 + e_4\ot e_4 \\
   &\quad+\frac12 (e_5\ot e_5 + e_6\ot e_6 + e_7\ot e_7 + e_8\ot e_8)\\
\cop(e_2)&= e_1\ot e_2 + e_2\ot e_1 + e_3\ot e_4 + e_4\ot e_3 \\
   &\quad+\frac12 (e_5\ot e_6 + e_6\ot e_5 + i e_7\ot e_8 - i e_8\ot e_7)\\
\cop(e_3)&= e_1\ot e_3 + e_3\ot e_1 + e_2\ot e_4 + e_4\ot e_2 \\
   &\quad+\frac12 (e_5\ot e_6 + e_6\ot e_5 - i e_7\ot e_8 + i e_8\ot e_7)\\
\cop(e_4)&= e_1\ot e_4 + e_4\ot e_1 + e_2\ot e_3 + e_3\ot e_2 \\
   &\quad+\frac12 (e_5\ot e_5 + e_6\ot e_6 - e_7\ot e_7 - e_8\ot
e_8) \displaybreak[0]\\ 
\cop(e_5)&= e_1\ot e_5 + e_5\ot e_1 + e_2\ot e_6 + e_6\ot e_2 \\
   &\quad+ e_3\ot e_6 + e_6\ot e_3 + e_4\ot e_5 + e_5\ot e_4\\
\cop(e_6)&= e_1\ot e_6 + e_6\ot e_1 + e_2\ot e_5 + e_5\ot e_2 \\
   &\quad+ e_3\ot e_5 + e_5\ot e_3 + e_4\ot e_6 + e_6\ot e_4\\
\cop(e_7)&= e_1\ot e_7 + e_7\ot e_1 -ie_2\ot e_8 +ie_8\ot e_2 \\
   &\quad+ ie_3\ot e_8 -ie_8\ot e_3 - e_4\ot e_7 - e_7\ot e_4\\
\cop(e_8)&= e_1\ot e_8 + e_8\ot e_1 +ie_2\ot e_7 -ie_7\ot e_2 \\
   &\quad- ie_3\ot e_7 +ie_7\ot e_3 - e_4\ot e_8 - e_8\ot e_4.
\end{align*}
Pal's idempotent state is defined by 
\[
\omega\biggl(\sum_{k=1}^8 \alpha_k e_k\biggr) = 
\frac14 \alpha_1 + \frac14 \alpha_4 + \frac12 \alpha_6. 
\]
As we shall see in section~\ref{sec:inv-id}, 
we can always associate a certain C*-subalgebra to an 
idempotent state. For Pal's idempotent state $\omega$,
the associated C*-subalgebra $(\id\ot \omega)\cop(\C(\G))$ 
is spanned by the elements 
\[
a = e_1 + e_2 + e_3 + e_4 + e_5 + e_6
\quad\text{and}\quad
b = e_1 - e_2 - e_3 + e_4 - e_5 + e_6.
\]
Moreover, one can calculate that 
\begin{equation}\label{eq:cop}
\cop(a) = a\ot a \quad\text{and}\quad
\cop(b) = b\ot b.
\end{equation} 
Now consider the quantum group $\dual\Z_2$ given by the group
C*-algebra $\C^*(\Z_2)$ (of course $\Z_2\cong \dual\Z_2$ but 
the chosen viewpoint suits us better).
Then $\C^*(\Z_2)$ is spanned by $\lambda(0)$ 
and $\lambda(1)$, where $\lambda$ denotes the left regular
representation of $\Z_2$. Define $\pi\col \C^*(\Z_2)\to \C(\G)$ by 
$\pi(\lambda(0)) = a$ and $\pi(\lambda(1)) = b$. 
By \eqref{eq:cop}, we see that $\pi$ preserves the quantum group
structure of $\dual\Z_2$. 
There is also a conditional expectation $E$ onto $\pi(\C^*(\Z_2))$ 
defined by 
\begin{gather*}
E(e_1) = E(e_4) = \frac18(a+b) \qquad E(e_2) = E(e_3) = \frac18(a- b) \\
E(e_5) = \frac14 (a-b) \qquad E(e_6) = \frac14 (a+b) \qquad E(e_7) = E(e_8) = 0.
\end{gather*}
The counit $\epsilon_{\dual\Z_2}$ of $\dual\Z_2$ is the constant function $1$ 
(considered as an element of $\B(\Z_2)$).
Finally, note that $\omega = \epsilon_{\dual\Z_2}\circ \pi\inv\circ E$.
What this shows is that Pal's idempotent state is 
of the similar form as the idempotent states $1_H$ on group C*-algebras:
$1_H = \epsilon_{\dual H} \circ \pi\inv \circ E$ 
where $\epsilon_{\dual H}$ is the counit of $\dual H$ (i.e.\ constant
$1$ on $H$),  $\pi\col \C_r^*(H)\to \C_r^*(G)$ is the natural
embedding (i.e.\ zero extension), and $E\col \C_r^*(G)\to
\pi(\C_r^*(H))$ is the natural conditional expectation
(i.e.\ restriction to $H$). 
So although Pal's idempotent is not like the idempotent states in the
commutative case (i.e.\ not a Haar idempotent),
it is similar to the idempotent states in the co-commutative case. 
Thus it is associated with a subgroup but in a different way.

These examples of idempotent states that are not Haar idempotents
show that a new approach is needed for general locally compact 
quantum groups. In section~\ref{sec:inv} we consider 
another notion, that of  left invariant C*-subalgebras,
that is closely tied with idempotent states as we shall see.
There is also the approach of Franz and Skalski,
who show in \cite{franz-skalski:algebraic-idem} 
that every idempotent state on a finite quantum group
arises from the Haar state of a so-called quantum
sub\emph{hyper}group. 
On a related note, Franz and Skalski \cite{franz-skalski:idempotent-states}
also show that idempotent states on a finite quantum group
correspond to quantum pre-subgroups in the sense of
\cite{baaj-blanchard-skandalis}.
In all these approaches one associates idempotent states 
with structures more general than subgroups, and that is what we shall 
do with left invariant C*-subalgebras in section~\ref{sec:inv-id}.

\section{Positive examples from deformation quantum groups}
\label{sec:SU}

In this section we shall consider idempotent states on 
some important examples of compact quantum groups,
in particular on the quantum deformation of $\SU(2)$
introduced by  Woronowicz \cite{woronowicz:SUq,
woronowicz:compact-matrix-pseudogroups}.
It turns out that on these deformations of classical groups,
$\SU_q(2)$, $\U_q(2)$ and $\SO_q(3)$, all idempotent states are Haar
idempotents.  The results in this section are due to 
Franz, Skalski and Tomatsu \cite{franz-skalski-tomatsu:idem}.

Define $\C(\SU_q(2))$ as the universal unital C*-algebra generated by
elements $a$ and $c$ such that 
\begin{gather*}
\begin{pmatrix}
a & -qc^*\\
c & a^* 
\end{pmatrix}
\end{gather*}
is formally a unitary matrix (a $2\times 2$ matrix with entries 
in $\C(\SU_q(2))$). The comultiplication  of $\SU_q(2)$ is determined by 
the identity
\[
\cop\begin{pmatrix}
a & -qc^*\\
c & a^* 
\end{pmatrix} =
\begin{pmatrix}
a & -qc^*\\
c & a^* 
\end{pmatrix} \otimes
\begin{pmatrix}
a & -qc^*\\
c & a^* 
\end{pmatrix}.
\]
This identity is to be read as follows: on the left-hand side 
we apply $\cop$ to each entry and on the right-hand side 
we take a formal matrix multiplication where we use the tensor 
product when `multiplying' entries; then we just equate the entries
of the two $2\times 2$ matrices. 
Although the relations given above fully determine the structure 
of $\SU_q(2)$, to  prove 
that we actually get a compact quantum group takes some work. 
Recall however that the existence of the Haar state follows from the
general theory of compact quantum groups.

Using the representation theory of $\SU_q(2)$,
Franz, Skalski and Tomatsu \cite{franz-skalski-tomatsu:idem}
calculated all the idempotent states 
on  $\SU_q(2)$ for $q \in (-1,0)\cup(0,1)$.
It turns out that these are all Haar idempotents. 
Namely, the idempotent states on $\SU_q(2)$ 
are the Haar state and the Haar idempotents coming from 
the subgroups $\T$ and $\Z_n$, $1\le n < \infty$.  
We see that $\T$ is a subgroup of $\SU_q(2)$ 
by mapping the generator $a$ to the generator $z$ of $\C(\T)$ 
and $c$ to $0$. Moreover, $\Z_n$'s are subgroups of $\T$. 
Already Podle\'s \cite{podles:SU-subgroups} showed 
that these are all the closed quantum subgroups of $\SU_q(2)$. 
This result follows also from \cite{franz-skalski-tomatsu:idem},
but of course it takes more work to show that all idempotent states
actually arise from these subgroups.
Franz, Skalski and Tomatsu also give the complete list of idempotent 
states for the related deformation quantum groups $\U_q(2)$ and 
$\SO_q(3)$. Also in these cases all idempotent states are Haar
idempotents.

\section{Left invariant C*-subalgebras} \label{sec:inv}

In this section we shall consider another notion related 
to idempotent states besides subgroups. 
The notion is that of left invariant C*-subalgebra
(here we could use alternative terminology 
and call these coideals or homogeneous spaces).

Let $\G$ be a coamenable locally compact quantum group.
For $\omega\in\C_0(\G)^*$, define the left and right convolution
operators on $\C_0(\G)$ by
\[
\begin{split}
\lt_\omega(a) &= (\omega\ot\id)\cop(a) \\
\rt_\omega(a) &= (\id\ot \omega)\cop(a) 
\end{split}
\qquad (a\in \C_0(\G)).
\]
A C*-subalgebra $X \sub \C_0(\G)$ is said to be \emph{left invariant}
if $\lt_{\omega}(X) \sub X$ for all $\omega \in \C_0(\G)^*$. 
A nondegenerate C*-subalgebra $X$ of $\C_0(\G)$ 
is left invariant if and only if $\cop\col X \to \M(\C_0(\G)\ot X)$.
(A C*-subalgebra is nondegenerate if it contains a bounded approximate 
identity for the ambient C*-algebra.)

Consider the commutative case when $G$ is a locally compact group.
Then a C*-subalgebra $X$ of $\C_0(G)$ is left 
invariant if and only if it is left translation invariant;
that is, the function $\lt_s f(t) = f(st)$ is in $X$ 
whenever $f\in X$ and $s\in G$. 
Lau and Losert \cite{lau-losert:complemented} 
have characterised left invariant C*-subalgebras 
of $\C_0(G)$: a C*-subalgebra $X\sub \C_0(G)$ 
is left invariant if and only if 
there is a compact subgroup $H$ of $G$ such that 
$X$ consists of all the functions in $\C_0(G)$ 
that are constant on right cosets of $H$.
The latter statement means that 
$X$ is $*$-isomorphic to $\C_0(G/H)$.
Earlier, de~Leeuw and Mirkil \cite{deleeuw-mirkil} 
gave this characterisation 
in the case of locally compact \emph{abelian} groups.
Moreover, Takesaki and Tatsuuma \cite{takesaki-tatsuuma:duality-subgroups}
produced several related results, 
characterising closed (left) invariant self-adjoint subalgebras of 
$\linfty(G)$, the Fourier algebra $\A(G)$, the group von Neumann algebra
$\vn(G)$ and the $\lone$ group algebra
(here the meaning of `closed' depends on the context:
with $\A(G)$ and $\lone(G)$ it means norm-closed and 
with $\linfty(G)$ and $\vn(G)$ it means weak*-closed).
The dual version of the Lau--Losert characterisation 
for a locally compact amenable group $G$ is given in
\cite{salmi:compact-subgroups}: 
a C*-subalgebra $X\sub \cstar_r(G)$ is invariant if and only if 
$X \cong \cstar_r(H)$ for some open subgroup $H$.
One can also consider strictly closed left invariant C*-subalgebras 
of $\C_b(G)$ and $\M(\cstar_r(G))$ and obtain in both cases a one-to-one
correspondence with closed subgroups of $G$ \cite{salmi:strict-invariant}.

Finally, we also have the following result from
\cite{salmi:compact-subgroups}, concerning left invariant
C*-subalgebras of coamenable locally compact quantum groups. 
Recall that a conditional expectation on a C*-algebra $A$  
is a norm $1$ projection from $A$ onto a C*-subalgebra of $A$.
The following result also employs a symmetry condition 
that is related to the problem brought out by Pal's counterexample. 
We postpone the formulation of this symmetry condition 
until after the theorem.

\begin{theorem} \label{thm:X-H}
There is a one-to-one correspondence between compact quantum subgroups
of $\G$ and symmetric, left invariant C*-subalgebras $X$ of $\C_0(\G)$  
with a conditional expectation $P$ from $\C_0(\G)$ onto $X$ such that
$(\id\ot P)\cop = \cop P$. 
\end{theorem}

Let $G$ be an amenable locally compact group and $H$ an 
open subgroup of $G$. As noted above,  $\cstar_r(H)$
is an invariant C*-subalgebra of $\cstar_r(G)$. However,
$\cstar_r(H)$ is not associated with a compact quantum subgroup 
unless $H$ is normal, in which case $\cstar_r(H)$ is associated with 
$\cstar_r(G/H)$. We shall need an analogue of this normality condition
for more  general quantum groups. This can be done through the 
so-called multiplicative unitary of a locally compact quantum group
$\G$.  There is a canonical way to define a unitary operator 
$W$ on $\ltwo(\G)\ot\ltwo(\G)$ such that $W$ determines 
the quantum group $\G$ \cite[Proposition~3.17]{kustermans-vaes:lcqg}.
Here $\ltwo(\G)$ denotes the Hilbert space obtained by applying 
the GNS-construction to the left Haar weight of $\G$. 
The C*-algebra $\C_0(\G)$ is faithfully represented on $\ltwo(\G)$ 
and it is natural to identify $\C_0(\G)$ with its image in $\B(\ltwo(\G))$.
The multiplicative unitary $W$ determines the comultiplication via
\[
\cop(a) = W^*(1\ot a)W \qquad(a\in \C_0(\G)).
\]
The notion of multiplicative unitary is central in the 
theory of locally compact quantum groups; seminal work 
in this area is due to Baaj and Skandalis~\cite{baaj-skandalis} and
Woronowicz~\cite{woronowicz:manageable}.

We say that a C*-subalgebra $X$ of $\C_0(\G)$  is \emph{symmetric}
if 
\[
W(x\ot 1)W^* \in \M(X\ot B_0(\ltwo(\G))) 
\]
whenever $x\in X$ (here $\B_0$ denotes the  compact operators).
Tomatsu \cite{tomatsu:coideals} introduced this type 
of condition, calling it coaction symmetry 
(due to the fact that $X$ is symmetric if and only if 
it is closed under the natural 
action of the dual quantum group of $\G$).
Returning to the co-commutative case,  
the left invariant C*-subalgebra $\C^*_r(H)$ 
associated with an open subgroup $H$ of $G$ is symmetric 
if and only if $H$ is normal \cite{salmi:compact-subgroups}.

\section{Idempotent states and left invariant C*-subalgebras} 
\label{sec:inv-id}

Although Pal's counterexample showed that we cannot associate 
all idempotent states to compact quantum subgroups, we may still have a 
chance of associating idempotent states to suitable left invariant 
C*-subalgebras. The results in this section are from 
\cite{salmi-skalski:idempotent-states}, many of them 
generalisations from \cite{franz-skalski:idempotent-states}
or \cite{franz-skalski-tomatsu:idem}.

Let $\G$ be a coamenable locally compact quantum group.
If $\omega$ is an idempotent state on $\G$, then 
$\rt_\omega(\C_0(\G))$ is a left invariant C*-subalgebra 
of $\C_0(\G)$ and $\rt_\omega$ is a conditional expectation 
onto this C*-subalgebra.  The following result 
generalises an earlier result due to Franz and Skalski 
\cite{franz-skalski:idempotent-states} concerning compact quantum groups.

\begin{theorem}
Suppose that $\G$ is unimodular \textup{(}i.e.\ $\phi=\psi$\textup{)}.
There is a one-to-one correspondence between
idempotent states $\omega$ on $\G$ and left invariant 
C*-subalgebras $X$ of $\C_0(\G)$
with a conditional expectation $P$ from $\C_0(\G)$ onto $X$  
such that $\phi\circ P = \phi$. The correspondence is given by
\[
X_\omega = \rt_\omega(\C_0(\G)),\qquad \omega_X = \epsilon P_X.
\]
where $\epsilon$ is the counit of $\G$.
\end{theorem}

The preceding result leaves room for improvement: 
one would like to remove the unimodularity assumption 
in which case the conditional expectation should preserve 
both left and right Haar weights. 

The following result characterises those idempotent states 
that arise from compact quantum subgroups. 
It also brings together the symmetry condition from 
the preceding section. 
The equivalence between 
the first and the third condition is proved for 
compact quantum groups in \cite{franz-skalski-tomatsu:idem}.

\begin{theorem} 
Let $\omega$ be an idempotent state on $\G$ and let $X_{\omega} =
\rt_\omega(\C_0(\G))$. The following are equivalent:
\begin{enumerate}
\item $\omega$ is a Haar idempotent; 
\item $X_{\omega}$ is symmetric;
\item $N_{\omega}:=\set{a \in \C_0(\G)}{\omega(a^*a)=0}$ is an ideal.
\end{enumerate}
\end{theorem}

The set $N_\omega$ in the third condition of the 
the preceding result is always a left ideal, so 
the condition is automatically satisfied if $\C_0(\G)$ is commutative. 
Consequently, we get the Kawada--It\^o 
theorem from section~\ref{sec:kawada-ito}
as a corollary. 

\begin{corollary}[Kawada--It\^o]
If $\C_0(\G)$ is commutative, then all
idempotent states on $\G$ are Haar idempotents.
\end{corollary}

Finally, we have the following correspondence result, 
which does not assume unimodularity, but works only 
for Haar idempotents.

\begin{theorem}
There is a one-to-one correspondence between
Haar idempotents $\omega$ on $\G$ and symmetric, left invariant 
C*-subalgebras $X$ of $\C_0(\G)$
with a conditional expectation $P$ from $\C_0(\G)$ onto $X$  
such that $\phi\circ P = \phi$ and $\psi\circ P = \psi$.
\end{theorem}

Note that the preceding theorem improves Theorem~\ref{thm:X-H}
in the sense that the condition that $(\id\ot P)\cop = \cop P$ 
may be replaced by the more natural condition that $P$ preserves 
both left and right Haar weights.

\begin{acknow}
This paper is based on a talk given in 
the conference `Operator Theory and its Applications'
held in honour of Victor Shulman in Gothenburg 2011;
I thank the organisers Ivan Todorov and Lyudmila Turowska
for a great conference. I thank Nico Spronk for generous 
support throughout my postdoctoral stay at University of Waterloo
and in particular for enabling my conference visit. 
I thank Emil Aaltonen Foundation for support during the
preparation of this paper. 
I thank Adam Skalski and the referee for helpful comments improving the paper. 
\end{acknow}

\end{document}